\documentclass[12pt]{amsart}

\usepackage{amsfonts,latexsym,amsmath,amsthm,amssymb, color}
\usepackage[margin=1.0in]{geometry}

\def\H{{\mathbb H}}

\def\C{{\mathbb C}}

\def\R{{\mathbb R}}

\newcommand{\diag}{\mathsf{diag} }
\newcommand{\eig}{\mathsf{eig}}
\newcommand{\per}{\mathsf{per}}
\newcommand{\sort}{\mathsf{sort}}

\newcommand{\scal}[1]{\langle#1\rangle}

\newtheorem*{theorem*}{Theorem}
\numberwithin{equation}{section}

\newtheorem{theorem}{Theorem}[section]
\newtheorem{proposition}[theorem]{Proposition}
\newtheorem{lemma}[theorem]{Lemma}
\newtheorem{corollary}[theorem]{Corollary}
\newtheorem{definition}  [theorem] {Definition}

\begin{document}

\title[A Daleskii-Krein theorem for Hermitian...]{A Daleskii-Krein theorem for Hermitian matrix-functions based on vector-fields}
\author{Marcus Carlsson}
\address{Centre for Mathematical Sciences, Lund University\\Box 118, SE-22100, Lund,  Sweden\\}
\email{marcus.carlsson@math.lu.se}

\begin{abstract}
We consider ``spectral'' matrix-functions for Hermitian matrices, where the novelty is that the function applied to the spectrum is allowed to be a vector-field rather than a scalar function (a.k.a isotropic matrix functions). We prove first order approximation formulas, generalizing the classical Daleskii-Krein theorem, as well as Lipschitz estimates.
\end{abstract}
\maketitle


\section{Introduction}

``Isotropic'' matrix functions arise naturally in mechanics of elastic materials, (see e.g.~\cite{silhavy2013mechanics}), where \textit{isotropic} refers to the fact that such a material behaves ``the same in all directions''. If $\H_d$ denotes the Hilbert spaces of all Hermitian $d\times d$ matrices endowed with the Frobenius norm, then $\mathcal{F}:\H_d\rightarrow\H_d$ is isotropic if \begin{equation}\label{tf}U^*\mathcal F(A)U=\mathcal F(U^*AU)\end{equation} for all unitary matrices $U$. Let $A=U_A \Lambda_\alpha U_A^{*}$ be the spectral decomposition of the matrix $A$, where $\alpha$ denotes the (non-increasingly ordered) eigenvalues of $A$ and $\Lambda_\alpha$ is the corresponding diagonal matrix. Given an isotropic matrix function $\mathcal F$, it is not hard to see that  \begin{equation}\label{funcvector}\mathcal F(A)=U_A\Lambda_{F(\alpha)}U_A^{*}:=\mathcal{L}_F(A)\end{equation} for some vector field $F$ on $\R^d$ satisfying certain restrictions. More precisely if we set $F(\alpha)=\mathcal F(\Lambda_{\alpha})$ then \eqref{funcvector} holds, and one easily sees that $F$ needs to be \textit{block-constant}, i.e.~such that $F_m(x)=F_n(x)$ whenever $x_m=x_n$. In this note we take the (block-constant) function $F$ as given and define $\mathcal{F}(A)=\mathcal{L}_F(A)$ via \eqref{funcvector}. We will for the remainder use the latter notation since it makes the link to the underlying vector field $F$ more explicit. Note that for this to be well defined it suffices to assume that the domain of $F$ is $\R^d_{\geq}$, where  $\R^d_{\geq}\subset\R^d$ denotes the set of non-increasing $d$-tuples. 

These matrix functions have  recently also started to be of interest in matrix optimization theory \cite{andersson-etal-ol-2017,grussler2018low,larsson-olsson-ijcv-2016,mousavi2018unified}. In fact, their study in the scalar case has been of interest to the optimization community for quite some time (see e.g.~\cite{lewis1996derivatives}), although the terminology has been rather different and it seems the two fields were disconnected up until recently, see \cite{mousavi2018unified} which also contains a long list of applications in various other areas.

The first order perturbation of this matrix function is completely understood by the so called Daleskii-Krein theorem, which shows that $\mathcal{L}_F$ is Fr\'{e}chet differentiable and gives a concrete formula for its Fr\'{e}chet derivative $\mathcal{L}_F'$, so that \begin{equation}\label{pert}\mathcal{L}_F(A+E)=\mathcal{L}_F(A)+\mathcal{L}_F'(E)+o(\|E\|).\end{equation}
(Since all norms on finite dimensional spaces are equivalent, it is not important to specify which norm we use, but often we work with the Frobenius-norm $\|E\|_2$ for simplicity. The expression $\|E\|$ with no subindex will be reserved for the operator norm.) The formula \eqref{pert} was first shown by J.~L.~Daleckii and M.~G.~Krein in the scalar case in \cite{daletskii1965integration}, and the isotropic extension was first shown by J.~Sylvester \cite{sylvester1985differentiability}. We here provide a seemingly new proof based on complex analytic tools. Based on this we then move on to establish Lipschitz estimates. More precisely we show that \begin{equation}\label{liplok}\|\mathcal{L}_F(A)-\mathcal{L}_F(B)\|_2\leq \|F\|_{Lip}\|A-B\|_2,\end{equation} which is the best estimate one could hope for (in terms of the Frobenius norm), shown in Section \ref{lipshitz}. This result seems to be new, a proof with a suboptimal constant appears in \cite{mohebi2007analysis}.

Interesting examples of vector-fields $F$ arise as proximal operators in low-rank approximation theory \cite{andersson-etal-ol-2017,grussler2018low,larsson-olsson-ijcv-2016}. It is our hope that the theory provided here can be used for faster evaluation of such operators in iterative algorithms, but this has to be investigated elsewhere. In this case, the Lipschitz-estimate is however of lesser use, since proximal operators are known to be firmly-nonexpansive (see e.g.~Theorem 21.2 and Corollary 23.8 in \cite{bauschke2017convex}), and the estimate \eqref{liplok} usually only gives nonexpansiveness.

As a final remark, we note that there exists corresponding scalar matrix-functions based on singular values rather than eigenvalues. A Daleskii-Krein formula has recently been proven also for this case in \cite{noferini2017formula}, and the corresponding Lipschitz estimate was shown in \cite{andersson-etal-pams-2016}. Fast computation of this type of matrix functions is considered in \cite{arrigo2016computation}, and in a future publication, based on the present work, we will extend Noferini's formula to the vector-field setting.

\section{Perturbation theory for $\mathcal{L}_F$}\label{sec1}

We now analyze how perturbations affect the functional calculus $\mathcal{L}_F$, where $F$ is any given function on $\R_{\geq}^d$, i.e.~we are interested in $\mathcal{L}_F(A+E)$ for small $E$. More precisely we shall analyze the Fr\'{e}chet derivative of this map. We of course assume that $F$ is block constant at $\alpha$, but not necessarily in the whole of $\R^d_{\geq}$. For example, $F$ could be $F(x)=(1,0,\ldots,0)$, and in this case $\mathcal{L}_F(A)$ equals the orthogonal projection onto the subspace spanned by the first eigenvector, which is well defined as long as $\alpha_1$ has multiplicity 1. Matrices in $\H_d$ may of course have complex off-diagonal entries, but we will treat $\H_d$ as a real normed vector space, which is important for the definition of the Fr\'{e}chet derivative.
Note that \begin{equation}\label{tbeg}\mathcal{L}_F(A+E)-\mathcal{L}_F(A)=U_A\Big(\mathcal{L}_F(\Lambda_\alpha+U_A^*EU)-\mathcal{L}_F(\Lambda_\alpha)\Big)U_A^*,\end{equation} by which it follows that it suffices to compute the derivative in the case when $A$ is diagonal. The matrix $U_A^*EU_A$ will henceforth be denoted $\hat E$ (in analogy with \cite{carlsson2018perturbation1,carlsson2018perturbation2}).

\subsection{Point-symmetric functions and vector-fields}
We first introduce the function class for which $\mathcal{L}_F$ is Fr\'{e}chet differentiable. Given a vector $x\in\R^d$, we let $\per(x)$ be the set of permutation matrices $\Pi$ such that $\Pi x=x$. Any vector-field $F$ that is differentiable at a point $x$ gives rise to a (matrix) derivative $F'|_{x}$, and the property that \begin{equation}\label{pe}\Pi^*F'|_{x}\Pi=F'|_{x},\quad \Pi\in \per(x),\end{equation} will turn out to be crucial. To simplify verification of this fact, we introduce ``point-symmetric functions'', where the terminology is adopted from \cite{sendov2007higher}.

Since we are interested in vector-fields that act on eigenvalues, it is natural to restrict attention to vector-fields $F$ defined only on $\R^d_\geq$. However, such vector-fields are not differentiable on the boundary, i.e.~at $d$-tuples with components of higher multiplicity than one. We now describe an elegant extension process which naturally leads to extensions $F^{ext}$ that satisfy \eqref{pe}.

Let $\sort(x)$ be the set of permutation matrices $\Sigma$ such that $\Sigma x\in\R^d_{\geq}$. Given a fixed $\Sigma\in \sort(x)$, note that \begin{equation}\label{sort}\sort(x)=\Sigma \cdot \per (x).\end{equation}
If $F:\R^d_\geq\rightarrow\R^d$ is block-constant, we can therefore uniquely extend it to a function on $\R^d$ by setting \begin{equation}\label{ext}F^{ext}(x)=\Sigma^*F(\Sigma x).\end{equation} To see this, we use \eqref{sort} and note that $\Sigma^*$ is the inverse of $\Sigma$.\footnote{If a vector-field $F$ is block constant on some subset of $\R^d_\geq$, then we can clearly define $F^{ext}$ as above on a subset of $\R^d$. We omit the details of this rather trivial consideration.
} As a simple example, consider the vector-field $F:\R_{\geq}^2\rightarrow \R^2$ defined by $F(x_1,x_2)=(x_1-x_2,0)$. The the extension to $x_1<x_2$ via \eqref{ext} then becomes $F^{ext}(x_1,x_2)=(0,x_2-x_1)$.

Following \cite{sendov2007higher}, we denote by $T^{k,d}$ the set of $k-$tensors on $\R^d$. The set $T^{0,d}$ is defined as $\R$, $T^{1,d}$ is readily identified with $\R^d$ and $T^{2,d}$ with the set of $d\times d$-matrices over $\R$. A function $f:\R^d\rightarrow\R$ is thus identified with a $T^{0,d}-$valued map, a vector-field with a $T^{1,d}-$valued map and so on.

A $k$-tensor valued map $f:\R^d\rightarrow T^{k,d}$ is called point-symmetric if \begin{equation}\label{point} f(x)[h_1,\ldots,h_k]=f(\Pi x)[\Pi h_1,\ldots,\Pi h_k]\end{equation}
for all permutation-matrices $\Pi$. Note in particular that a point symmetric map satisfies $f(x)[h_1,\ldots,h_k]=f(x)[\Pi h_1,\ldots,\Pi h_k]$ whenever $\Pi\in \per(x)$.
This means that if $k=0$, point symmetric functions coincides with ``symmetric functions'' as defined e.g.~in \cite{lewis1996derivatives}, and it is easily seen that a block-constant vector-field $F:\R_\geq^d\rightarrow\R^d$ gives rise to a point symmetric extension $F^{ext}$.

A key feature of point symmetric maps is that the property is invariant under differentiation. More precisely, a Fr\'{e}chet derivative of a $T^{k,d}$-valued map naturally identifies with a new $T^{k+1,d}$-valued map, and if the former is point-symmetric then so is the latter. This is easy to show, we refer to \cite{sendov2007higher} for the details. In particular, if we identify a tensor $\mathcal{T}\in T^{2,d}$ with an $d\times d$-matrix ${M}$ as usual (i.e.~by the formula $\mathcal{T}[h_1,h_2]=h_2^tM h_1$) then $\mathcal{T}(\Pi h_1,\Pi h_2)$ translates to the matrix $\Pi^* M\Pi$. Summing up, we have proved the following:

\begin{proposition}\label{late}
If a vector-field $F$ is block-constant at $x$ then $F^{ext}$ is point-symmetric there. If $F^{ext}$ is also $C^1$ at $x$, then $(F^{ext})'|_x$ satisfies \eqref{pe}.
\end{proposition}

Note that if $F$ is only defined on $\R^d_{\geq}$ then it is impossible to compute all partial derivatives at any point $x$ with components of multiplicity greater than one, hence it is crucial to use $F^{ext}$ and not simply $F$ in the last statement. Since \eqref{pe} is of key importance, let us illustrate what it entails in practice and introduce some notation. Given any $\alpha\in \R^d$ we let $\tilde{\alpha}_1,\ldots,\tilde\alpha_k$ be a (non-increasing) enumeration of the distinct values of $\alpha$. Given $1\leq\tilde m\leq k$ we introduce $S_{\tilde m}=\{m:\alpha_m=\tilde\alpha_{\tilde m}\}$, i.e.~ $S_{\tilde m}$ is the ``block'' corresponding to the value $\tilde{\alpha}_{\tilde{m}}$ in $\alpha$. Clearly $F$ is block-constant at $\alpha$ if and only if there are numbers $s_1,\ldots,s_k$ such that $F_{m}(\alpha)=s_{\tilde m}$ for all $m\in S_{\tilde{m}}$ and all $1\leq\tilde m\leq k$.
Let $\bar{1}$ denote the matrix containing only ones, where the context determines its size. We now illustrate \eqref{pe} in a similar way.

\begin{proposition}\label{lok}
The identity \eqref{pe} holds at $x=\alpha$ if and only if for each pair $1\leq \tilde m,\tilde n\leq k$, there are numbers $r_{\tilde{m}}$ and $t_{\tilde m,\tilde n}$ such that the matrix $F'|_{\alpha}$ equals \begin{equation}\label{pe2}r_{\tilde m}I+t_{\tilde m,\tilde n}\bar{1}\end{equation} in the subblock with indices in $S_{\tilde m}\times S_{\tilde n}$.
\end{proposition}

Note that $t_{\tilde m,\tilde n}$ can be different from $t_{\tilde n,\tilde m}$ and that the diagonal term is only relevant when $\tilde m=\tilde n$. Moreover, if $|S_{\tilde m}|=1$ then $t_{\tilde m,\tilde m}$ is not well defined, so we set it to 0 by default.

\begin{proof}
Let $\Pi\in \per(\alpha)$ be arbitrary and, given any $1\leq\tilde m\leq k$, let $\Pi_{\tilde m}$ be the submatrix that permutes the elements in $S_{\tilde m}$. If $F'|_{\alpha,(\tilde m,\tilde n)}$ denotes the submatrix of $F'|_\alpha$ with indices in $S_{\tilde m}\times S_{\tilde n}$, then \eqref{pe} is equivalent to showing that \begin{equation}\label{pe1}\Pi^*_{\tilde m}F'|_{\alpha,(\tilde m,\tilde n)}\Pi_{\tilde n}=F'|_{\alpha,(\tilde m,\tilde n)}.\end{equation}
If $\tilde m\neq \tilde n$ then the left hand side can move any component of $F'|_{\alpha,(\tilde m,\tilde n)}$ to any given new position. Hence \eqref{pe1} is fulfilled if and only if $F'|_{\alpha,(\tilde m,\tilde n)}$ is constant. If $\tilde m\neq \tilde n$ then we have the same permutation matrix to the left and right in \eqref{pe1}, which means that the left hand side can move any diagonal component to a new \textit{diagonal} position, whereas the surrounding components can be moved arbitrarily outside of the diagonal. It follows that \eqref{pe1} is satisfied if and only if \eqref{pe2} holds.
\end{proof}

To sum up this section, we have shown that a vector-field $F$ that is initially only defined on $\R^d_{\geq}$ can be extended to $\R^d$ in a natural way such that the extension is point symmetric and the key equation \eqref{pe} holds whenever the extended vector-field $F^{ext}$ is $C^1$ there. Obviously, if $F$ is already defined on $\R^d$ and is point symmetric, we can do without the extension $F^{ext}$. Also, it is possible to do these definitions also locally, but for the sake of keeping notation simple we have omitted this.
\begin{definition}\label{point sym}
A vector field $F:\R^d_\geq\rightarrow\R^d$ is called $C^1$ point-symmetric at $\alpha\in\R^d_\geq$ if $F^{ext}$ is point-symmetric and $C^1$ at $\alpha$. If it is $C^1$ point-symmetric at all points in $\R^d_\geq$, we simply call it $C^1$ point-symmetric.
\end{definition}

The results of this paper pertain to matrix-functions that are $C^1$ point-symmetric, and therefore it is relevant to ask whether this class is larger than those already considered by Lewis and Sendov \cite{lewis2001twice}, i.e.~those arising as $F=\nabla f$ for some symmetric scalar function $f$. The answer is yes, for the latter class equals the set of \textit{conservative} vector-fields, i.e.~vector-fields $F$ such that $\partial_mF_n=\partial_nF_m$, and it is easy to see that the class of $C^1$ point-symmetric vector-fields is much larger. To see a simple example of this, consider the vector-field $F:\R_{\geq}^2\rightarrow \R^2$ defined by $F(x_1,x_2)=((x_1-x_2)^2,0)$. The the extension to $x_1<x_2$ via \eqref{ext} then becomes $F^{ext}(x_1,x_2)=(0,(x_1-x_2)^2)$ which clearly is $C^1$ also at the boundary $x_1=x_2$.


\subsection{The partial derivatives of $\mathcal{L}_F$}

We can now start to compute Gateux derivatives of $\mathcal{L}_F$ at $\Lambda_\alpha$ for certain directions $E$,  i.e.~the derivative of $\R\ni h\mapsto \mathcal{L}_F(\Lambda_\alpha+hE)$ at $h=0$, which we denote by $\frac{\partial\mathcal{L}_F}{\partial E}(\Lambda_\alpha)$. Let $E_{m,n}$ be the matrix which is 1 at indices $(m,n)$ and has zeroes elsewhere, and let $({e}_m)_{m=1}^d$ denote the canonical basis in $\C^d$. We will also use the notation $\diag(\alpha)$ for $\Lambda_\alpha$.

\begin{lemma}\label{l11}
If $F$ is a $C^1$ point-symmetric vector field at $\alpha$, then $$\frac{\partial\mathcal{L}_F}{\partial E_{m,m}}(\Lambda_\alpha)=\diag((F^{ext})'|_{\alpha}{e}_m).$$
\end{lemma}
\begin{proof}
Let $e\in \R^d$ be given and let $\Pi\in \per(\alpha)$ be such that $\alpha+h \Pi e$ is non-increasing,  for all small enough $h>0$ (we assume for the remainder that $h$ is arbitrary but positive). Note that this means that $\Pi$ only permutes within each subblock $S_m$, in which it orders the components of $h$ non-increasingly. If $E=\diag(e)$ then the spectral decomposition   of $\Lambda_\alpha+hE$ equals $\Pi^*\diag(\alpha+h \Pi e)\Pi$, so \begin{equation*}\frac{1}{h}\Big(\mathcal{L}_F(\Lambda_\alpha+hE)-\mathcal{L}_F(\Lambda_\alpha)\Big)=\frac{1}{h}\Big(\Pi^*\diag(F(\alpha+h \Pi e)-F(\alpha))\Pi\Big).\end{equation*} Upon taking a limit as $h\rightarrow 0^+$ we obtain $\Pi^*\diag((F^{ext})'|_{\alpha} \Pi e)\Pi=\diag(\Pi^{*} (F^{ext})'|_{\alpha}\Pi e)$. By Proposition \ref{late} this expression is independent of $\Pi$, and hence we can drop the assumption $h>0$. Taking the limit as $h\rightarrow 0$ we conclude that \begin{equation*}\lim_{h\rightarrow 0}\frac{1}{h}\Big(\mathcal{L}_F(\Lambda_\alpha+hE)-\mathcal{L}_F(\Lambda_\alpha)\Big)= (F^{ext})'|_{\alpha} e.\end{equation*}
 In particular, if $E=E_{m,m}$ we get the statement in the lemma.
\end{proof}

\textbf{Remark:} In order for $\mathcal{L}_F$ to be Fr\'{e}chet differentiable this expression $\diag(\Pi^{*} (F^{ext})'|_{\alpha}\Pi e)$ for the one-sided Gateaux derivative needs to be linear in $E$, which happens if and only if \eqref{pe} holds for all permutations $\Pi\in Per(\alpha)$. This condition is thus necessary for Frechet differentiability.


We now consider the off-diagonal components. We will consider $E_{m,n}+E_{n,m}$ and $iE_{m,n}-iE_{n,m}$, which combined with $E_{m,m}$ provides a natural basis for $\H_d$ as a real vector space. A perturbation in the direction of $E_{m,n}$ changes the eigenvalues $\alpha_m,\alpha_n$ in a nontrivial way (see e.g.~\cite{carlsson2018perturbation1}), but we basically consider a $2\times 2$ matrix problem since the others are unaffected by the perturbation, and hence the upcoming proof is basic. Given a $C^1$ point-symmetric vector-field $F$ we first introduce the notation \begin{equation}\label{Flambda}[F,\alpha]({m, n})=
\left\{\begin{array}{cc}
  \frac{F_m(\alpha)-F_n(\alpha)}{\alpha_m-\alpha_n} & \alpha_m\neq \alpha_n \\
  \partial_m F_m(\alpha)-\partial_m F_n(\alpha) & \alpha_m=\alpha_n,~m\neq n\\
  \partial_m F_m(\alpha) & m=n
\end{array}\right.
\end{equation} Note that there seems to be a lack of symmetry in $m$ and $n$ on the second line, but this is not so due to the assumption that $F$ is $C^1$ point-symmetric and Proposition \ref{lok}. In terms of the numbers introduced before Definition \ref{point sym} we have, for $m\in S_{\tilde m}$ and $n\in S_{\tilde n}$ that $[F,\alpha]({m,n})=
  \frac{s_{\tilde m}-s_{\tilde n}}{\tilde\alpha_{\tilde m}-\tilde\alpha_{\tilde n}}$ when $m$ and $n$ belong to different blocks, i.e.~when $\tilde m\neq \tilde n$, we have $[F,\alpha]({m,n})=r_{\tilde m}$ when $m\neq n$ belong to the same block, whereas
$[F,\alpha]({m,n})=  r_{\tilde m}+t_{\tilde m,\tilde m}$ on the diagonal $m=n$.

\begin{lemma}\label{l1}
Let $\tau\in\C$ be unimodular. If $F$ is $C^1$ point-symmetric at $\alpha$ and $m\neq n$, then $$\frac{\partial\mathcal{L}_F}{\partial (\tau E_{m,n}+\bar\tau E_{n,m})}(\Lambda_\alpha)=[F,\alpha]\circ (\tau E_{m,n}+\bar\tau E_{n,m}),$$
where $\circ$ denotes Hadamard multiplication of the matrices.
\end{lemma}
\begin{proof}
We may assume that $m=1$, $n=2$ and that $\Lambda_\alpha$ is a $2\times 2$ matrix, as noted earlier. Then $\Lambda_\alpha+h{(\tau E_{1,2}+\bar\tau E_{2,1})}=\left(
                           \begin{array}{cc}
                             \alpha_1 & \tau h \\
                             \bar \tau h & \alpha_2 \\
                           \end{array}
                         \right)$. If $\alpha_1\neq \alpha_2$, some simple calculations give $$\left(
                           \begin{array}{cc}
                             \alpha_1 & \tau h \\
                             \bar \tau h & \alpha_2 \\
                           \end{array}
                         \right)=\left(\left(
                           \begin{array}{cc}
                             1 & -\frac{\tau h}{\alpha_1-\alpha_2} \\
                             \frac{\bar \tau h}{\alpha_1-\alpha_2} & 1 \\
                           \end{array}
                         \right)+O(h^2)
\right)\Big(\Lambda_\alpha+O(h^2)\Big)\left(\left(
                           \begin{array}{cc}
                             1 & \frac{\tau h}{\alpha_1-\alpha_2} \\
                             -\frac{\bar \tau h}{\alpha_1-\alpha_2} & 1 \\
                           \end{array}
                         \right)+O(h^2)
\right)$$ which yields that
\begin{align*}&\frac{1}{h}\left(\mathcal{L}_F(\left(
                           \begin{array}{cc}
                             \alpha_1 & \bar \tau h \\
                             \bar \tau h & \alpha_2 \\
                           \end{array}
                         \right))-\mathcal{L}_F(\left(
                           \begin{array}{cc}
                             \alpha_1 & 0 \\
                             0 & \alpha_2 \\
                           \end{array}
                         \right))\right)=\frac{1}{h}\Big(\\&\left(
                           \begin{array}{cc}
                             1 & -\frac{\tau h}{\alpha_1-\alpha_2} \\
                             \frac{\bar \tau h}{\alpha_1-\alpha_2} & 1 \\
                           \end{array}
                         \right)\left(
                           \begin{array}{cc}
                             F_1(\alpha) & 0 \\
                             0 & F_2(\alpha) \\
                           \end{array}
                         \right)\left(
                           \begin{array}{cc}
                             1 & \frac{\tau h}{\alpha_1-\alpha_2} \\
                             -\frac{\bar \tau h}{\alpha_1-\alpha_2} & 1 \\
                           \end{array}
                         \right)-\left(
                           \begin{array}{cc}
                             F_1(\alpha) & 0 \\
                             0 & F_2(\alpha) \\
                           \end{array}
                         \right)+O(h^2)\Big)=\\&
                         \left(\begin{array}{cc}
                             0 & \frac{F_1(\alpha)-F_2(\alpha)}{\alpha_1-\alpha_2}\tau  \\
                             \frac{F_1(\alpha)-F_2(\alpha)}{\alpha_1-\alpha_2}\bar \tau  & 0 \\
                           \end{array}
                         \right)+O(h)=\left(\begin{array}{cc}
                             0 & {}[F,\alpha]\tau  \\
                             {}[F,\alpha]\bar \tau  & 0 \\
                           \end{array}
                         \right)+O(h),
\end{align*}
as desired. If $\alpha_1=\alpha_2$ we instead get $$\left(
                           \begin{array}{cc}
                             \alpha_1 & \tau h \\
                             \bar \tau h & \alpha_1 \\
                           \end{array}
                         \right)=\frac{1}{2}\left(\left(
                           \begin{array}{cc}
                             1 &  \tau  \\
                             \bar\tau & -1 \\
                           \end{array}
                         \right)
\right)\left(
                           \begin{array}{cc}
                             \alpha_1+ h & 0 \\
                             0 & \alpha_1- h \\
                           \end{array}
                         \right)\left(
                           \begin{array}{cc}
                             1 & \tau  \\
                             \bar \tau & -1 \\
                           \end{array}
                         \right)
$$ (without ordo terms). Since $F$ is assumed to be $C^1$ point-symmetric at $\alpha$, Proposition \ref{lok} gives that there are values $s,r,t$ such that $$F_1(\alpha_1+h,\alpha_1-h)=s+\scal{(r+t,t),(h,-h)}+o(h)=s+rh+o(h),$$ and similarly $F_2(\alpha_1+h,\alpha_1-h)=s+\scal{(t,r+t),(h,-h)}+o(h)=s-rh+o(h)$. Thus
\begin{align*}&\mathcal{L}_F(\left(
                           \begin{array}{cc}
                             \alpha_1 & \tau h \\
                             \bar \tau h & \alpha_1 \\
                           \end{array}
                         \right))=\frac{1}{2}\left(
                           \begin{array}{cc}
                             1 & \tau \\
                             \bar \tau  & -1 \\
                           \end{array}
                         \right)\left(
                           \begin{array}{cc}
                             F_1(\alpha_1+h,\alpha_1-h) & 0 \\
                             0 & F_2(\alpha_1+h,\alpha_1-h) \\
                           \end{array}
                         \right)\left(
                           \begin{array}{cc}
                             1 & \tau  \\
                             \bar \tau  & -1 \\
                           \end{array}
                         \right)=\\&\frac{1}{2}
                         \left(\begin{array}{cc}
                             (s+rh)+(s-rh) & ((s+rh)+(s-rh)) \tau  \\
                             ((s+rh)-(s-rh))\bar \tau  & (s+rh)+(s-rh) \\
                           \end{array}
                         \right)+o(h)=                         \left(\begin{array}{cc}
                             s & rh \tau  \\
                             rh \bar \tau  & s \\
                           \end{array}
                         \right)+o(h).
\end{align*}
It follows that
\begin{align*}&\mathcal{L}_F(\left(
                           \begin{array}{cc}
                             \alpha_1 & \tau h \\
                             \bar \tau h & \alpha_1 \\
                           \end{array}
                         \right))-\mathcal{L}_F(\left(
                           \begin{array}{cc}
                             \alpha_1 & 0 \\
                             0 & \alpha_1 \\
                           \end{array}
                         \right))=
                         h\left(\begin{array}{cc}
                             0 & r\tau  \\
                             r \bar \tau  & 0 \\
                           \end{array}
                         \right)+o(h)
\end{align*}
from which the desired conclusion easily follows.

\end{proof}

It is interesting to note that when $\alpha_1=\alpha_2$ the eigenvectors, i.e. $(1, \bar \tau)$ and $(\tau, 1)$, are discontinuous as functions of the perturbation $\tau E_{m,n}+\bar \tau E_{n,m}$. There is a wealth of literature investigating how instable the eigenspaces are, given a certain separation of the distinct eigenvalues in $\Lambda_\alpha$, see e.g. Ch. VII of \cite{bhatia2013matrix} or \cite{davis1970rotation} or \cite{carlsson2018perturbation1}. With this in mind, it is a bit surprising that $\mathcal{L}_F$ (which inherently relies on the eigenvectors) is differentiable, as we shall show next.

\subsection{Fr\'{e}chet differentiability of $\mathcal{L}_F$}

We recall that a function $\mathcal{F}$ on $\H_d$ is Fr\'{e}chet differentiable at $A\in\H_d$ if there exists a real-linear operator $\mathcal{F}':\H_d\rightarrow\H_d$ such that $$\mathcal{F}(A+E)=\mathcal{F}(A)+\mathcal{F}'(E)+o(\|E\|).$$

\begin{lemma}\label{l111}
Let $A\in \H_d$ be given and let $F,G:\R^d\rightarrow \C^d$ be block-constant in a neighborhood of $\alpha$ such that $F(x)=G(x)+o(\|x-\alpha\|_2)$.
Then $\mathcal{L}_F$ is Fr\'{e}chet differentiable at $A$ if and only if $\mathcal{L}_G$ is, and the derivatives coincide.
\end{lemma}
\begin{proof}
Let $\Lambda_\alpha+E=U\Lambda_\xi U^*$ be the spectral decomposition of $\Lambda_\alpha+E$. Since $\mathcal{L}_F(A)=\mathcal{L}_G(A)$ we get
\begin{align*}&\mathcal{L}_F(A+E)-\mathcal{L}_F(A)=U(\diag\big(F(\xi)-G(\xi)\big) )U^*+\mathcal{L}_G(A+E)-\mathcal{L}_G(A)=\\&O(F(\xi)-G(\xi))+\mathcal{L}_G(A+E)-\mathcal{L}_G(A)=\\&o(\|\xi-\alpha\|_2)+\mathcal{L}_G(A+E)-\mathcal{L}_G(A)=o(\|E\|_2)+\mathcal{L}_G(A+E)-\mathcal{L}_G(A)\end{align*}
where the last identity follows from the Hoffman-Wielandt inequality.
From this it is clear that Fr\'{e}chet differentiability of $G$ implies that of $F$ and vice versa.
\end{proof}

We remind the reader that $\tilde{\alpha}_1,\ldots,\tilde\alpha_k$ is a (non-increasing) enumeration of the distinct eigenvalues of $A,$ and that  $S_{\tilde m}=\{m:\alpha_m=\tilde\alpha_{\tilde m}\}$.
It is a basic result that $A$ can be written $$A=\sum_{{\tilde m}=1}^k\tilde\alpha_{{\tilde m}}P_{\tilde m}$$ where $P_{\tilde m}$ is the orthogonal projection onto the eigenspace corresponding to $\tilde\alpha_{\tilde m}$. This is also known as the spectral projection and is given by the formula \begin{equation}\label{fdx}P_{\tilde m}(A)=\int_{\Gamma_{\tilde m}}(zI-A)^{-1}\frac{dz}{2\pi i},\end{equation}
where $\Gamma_{\tilde m}$ is a suitably chosen circle around $\tilde\alpha_{\tilde m}$, which is shown in any book on functional analysis or operator theory.

\begin{proposition}\label{p1}
If $F$ is a $C^1$ point-symmetric vector-field in a neighborhood of $\alpha$, then $\mathcal{L}_F$ is Fr\'{e}chet differentiable.
\end{proposition}
\begin{proof}
By the lemma it suffices to prove the above statement for $$G(x)=F({\alpha})+F'|_{\alpha}(x-\alpha).$$ Let $B\in\H_d$ denote a matrix in the vicinity of $A$ and denote its eigenvalues by $\beta$. We then have
\begin{align*}&\mathcal{L}_G(B)=U_B(\diag(F(\alpha)) +\diag(F'|_{\alpha}(\beta-\alpha) )U_B^*=\\&
\sum_{{\tilde m}=1}^kU_B(\diag(F({\alpha})) +\diag(F'|_{\alpha}(\beta-\alpha) )U_B^*P_{\tilde m}(B)\end{align*} where diagonal values outside of the block $S_{\tilde m}$ are irrelevant for the ${\tilde m}$:th term. Therefore the computation can be continued as follows \begin{align*}
&\sum_{{\tilde m}=1}^kU_B\left(s_{\tilde m} I +r_{\tilde m}\diag(\beta-\alpha)+\Big(\sum_{{\tilde n}=1}^k t_{{\tilde m},{\tilde n}}\sum_{n\in S_{\tilde n}}(\beta_{n}-\tilde{\alpha}_{\tilde n})\Big)I\right)U_B^*P_{\tilde m}(B)=\\&
\sum_{{\tilde m}=1}^k r_{\tilde m} B P_{\tilde m}(B)+\sum_{{\tilde m}=1}^k\left(s_{\tilde m} -r_{\tilde m}\tilde\alpha_{\tilde m}+\Big(\sum_{{\tilde n}=1}^k t_{{\tilde m},{\tilde n}}\sum_{n\in S_{\tilde n}}(\beta_{n}-\tilde{\alpha}_{\tilde n})\Big)\right)P_{\tilde m}(B).
\end{align*}
Now note that we can write $B P_{\tilde m}(B)=\int_{\Gamma_{\tilde m}}\frac{z}{(zI-B)}\frac{dz}{2\pi i}$ and that $$\sum_{n\in S_{\tilde n}}(\beta_n-\tilde\alpha_{\tilde n})=\int_{\Gamma_{\tilde n}}\frac{(z-\tilde\alpha_{\tilde n})\frac{d}{dz}\det(zI-B)}{\det(zI-B)}\frac{dz}{2\pi i},$$ for $B$ sufficiently close to $A$. Summing up we have shown that\begin{equation}\label{hyf}
\begin{aligned}&\mathcal{L}_G(B)=\sum_{{\tilde m}=1}^k \int_{\Gamma_{\tilde m}}\frac{r_{\tilde m} z}{zI-B}\frac{dz}{2\pi i}+\\&\sum_{{\tilde m}=1}^k\left(s_{\tilde m} -r_{\tilde m}\tilde\alpha_{\tilde m}+\Big(\sum_{{\tilde n}=1}^k t_{{\tilde m},{\tilde n}}\int_{\Gamma_{\tilde n}}\frac{(z-\tilde\alpha_{\tilde n})\frac{d}{dz}\det(zI-B)}{\det(zI-B)}\frac{dz}{2\pi i}\Big)\right)\int_{\Gamma_{\tilde m}}(zI-B)^{-1}\frac{dz}{2\pi i}.
\end{aligned}\end{equation}
This calculation was made under the assumption that $B\in \H_d$ but the final expression is actually a valid expression for all matrices $B$, and as such it is holomorphic in each variable $B_{i,j}$ separately. If we identify the space of $d\times d$-matrices with $\C^{d^2}$, the function is then holomorphic and it is well known that this implies that the function is $C^\infty$, see e.g.~Theorem 2.2.1 and Corollary 2.2.2 in \cite{krantz2001function}. In particular it is differentiable, and hence the expression in \eqref{hyf}
is also Fr\'{e}chet differentiable on the space of $d\times d$-matrices. Since the restriction to the subspace $\H_d$ equals $\mathcal{L}_G(B)$ for $B$ in a neighborhood of $A$, it follows that $\mathcal{L}_G$ is Fr\'{e}chet differentiable at $A$.
\end{proof}

We now come to the first main theorem, a generalization of the so called Daleskii-Krein theorem \cite{bhatia2013matrix,daletskii1965integration,higham2008functions,roger1991topics} to the matrix-functions based on vector-fields introduced in this paper.  Given any matrix $M$ we use the notation $M^o=M-M\circ I$, i.e.~$M^o$ coincides with $M$ off the diagonal, whereas it is 0 on the diagonal.

\begin{theorem}\label{t1}
Let $A=U\Lambda_\alpha U^*$ be given and let $F$ be $C^1$ point-symmetric vector-field at $\alpha$. Given $E$ in $\H_d$, set $\hat E=U^*EU$ and let $\hat{e}$ be the vector with the diagonal components of $\hat E$.
Then $$\mathcal{L}_F'(E)=U\Big([F,\alpha]\circ \hat E+\diag_{(F'|_{\alpha})^o \hat{e}}\Big)U^*.$$
\end{theorem}

If $f$ is real valued and $F(x)=(f(x_1),\ldots,f(x_d))$, then $\mathcal{L}_F$ reduces to the traditional matrix functions (functional calculus) and the below theorem implies the Daleskii-Krein theorem, which we elaborate more on in Section \ref{tg}. If $F=\nabla f$ for some symmetric function $f:\R^d\rightarrow \R$, the theorem provides a formula for the second order term in a Taylor type expansion of $f(\eig(A+E))$, and the above theorem reduces to Theorem 3.2 of \cite{lewis2001twice}. The proof given here is more general and also shorter, it seems. Formulas for all possible orders was subsequently found in \cite{sendov2007higher}, but we will not pursue a similar quest here.

\begin{proof}
$\mathcal{L}_F$ is Fr\'{e}chet differentiable by the above proposition. Write $$\hat E=\sum_{m} a_m E_{m,m}+\sum_{n>m}b_{m,n}(E_{m,n}+E_{n,m})+\sum_{n>m}c_{m,n} i (E_{m,n}-E_{n,m})$$ and note that the above matrices provide an orthogonal basis for $\H_d$ as a real vector space. Using Lemma \ref{l11} for the first sum and \ref{l1} for the latter two (with $\tau$ equal to 1 and $i$ respectively), the formula $\mathcal{L}_F'(E)=U\Big([F,\alpha]^o\circ \hat E+\diag_{F'|_{\alpha} \hat{e}}\Big)U^*$ follows from basic multivariable calculus, and clearly $$[F,\alpha]^o\circ \hat E+\diag_{F'|_{\alpha} \hat{e}}=[F,\alpha]\circ \hat E+\diag_{(F'|_{\alpha})^o \hat{e}}.$$
\end{proof}

\subsection{The scalar case}\label{tg}

We now specialize to the scalar valued functional calculus, i.e.~when $F(x)=(f(x_1),\ldots, f(x_d))$ where $f:\R\rightarrow \R$ (or some subset including the actual spectrum). It is clear that any differentiable $f$ yields a $C^1$ point-symmetric $F$. Given $\alpha\in\R^d_{\geq}$, the matrix $[F,\alpha]$ will then be written $[f,\alpha]$ and simplifies to \begin{equation}\label{flambda}[f,\alpha]({i,j})=
\left\{\begin{array}{cc}
  \frac{f(\alpha_m)-f(\alpha_n)}{\alpha_m-\alpha_n} & \alpha_m\neq \alpha_n \\
  {f}'(\alpha_m) & \alpha_m=\alpha_n
\end{array}\right.
\end{equation} With this notation Theorem \ref{t1} implies the Daleskii-Krein theorem;
\begin{corollary}\label{c1}
Let $f$ be a $C^1$-function and $A=U_A\Lambda_\alpha U_A^*\in\H_d$. Given $E$ in $\H_d$ set $\hat E=U_A^*EU_A$. Then $$\mathcal{L}_f'(E)=U\Big([f,\lambda]\circ \hat E\Big)U^*.$$
\end{corollary}

If $A$ is invertible the above theorem can be used to approximate $\sqrt{A+E}$ for some small perturbation $E$, without having to do a spectral decomposition of $A+E$. It is interesting to note that $\sqrt{A+E}$ can be given a similar approximation even in the case when $A$ is singular, but this turns out to be a rather delicate issue, we refer to the adjacent paper \cite{carlsson2018perturbation2} for more on this particular case.

\section{Lipschitz continuity of $\mathcal{L}_F$}\label{lipshitz}

We can now prove that $\mathcal{L}_F$ is Lipschitz continuous (with respect to the Frobenius norm) with the same constant as $F$. We recall that for  $F:\R^d\rightarrow\R^d$ we
have that $$\|F\|_{Lip}=\sup_{x\neq y}\frac{\|F(x)-F(y)\|_2}{\|x-y\|_2},$$ which for $C^1$-functions equals $\sup_{x}\|F'|_{x}\|$ where the last norm refers to the operator norm
$$\|F'|_{x}\|=\sup_{y:~\|y\|=1}\frac{\|F'|_{x}y\|_2}{\|y\|_2}.$$
\begin{lemma}\label{rv}
Let $F$ be a $C^1$ point-symmetric vector-field at a particular point $\alpha$. Then $|[F,\alpha]({m,n})|\leq \|F\|_{Lip}$.
\end{lemma}
\begin{proof}
The statement is obvious if $m=n$. Suppose that $m>n$ and that they belong to the same block. Then $[F,\alpha]({m,n})=\partial_m F_m-\partial_m F_n$, $\partial_m F_m=\partial_n F_n$ and $\partial_m F_n=\partial_n F_m$ (see Proposition \ref{lok}). The inequality then follows by noting that $$|2[F,\alpha]({m,n})|=|\scal{e_m-e_n,F'|_{\alpha}(e_m-e_n)}|\leq \|F'|_{\alpha}\|\|e_m-e_n\|^2=2\|F\|_{Lip}.$$
Finally, if $m\in S_{\tilde m}$ and $n\in S_{\tilde n}$ with ${\tilde m}< {\tilde n}$, then we have to prove that \begin{equation}\label{gt6}|{s_{\tilde m}-s_{\tilde n}}|\leq \|F\|_{Lip}(\tilde\alpha_{\tilde m}-\tilde\alpha_{\tilde n})\end{equation} If we show that this holds for any two adjacent numbers, we can write $s_{\tilde m}-s_{\tilde n}=\sum_{k={\tilde m}}^{{\tilde n}-1} s_{k}-s_{k+1}$ and obtain \eqref{gt6} by using the triangle inequality. We thus assume that ${\tilde n}={\tilde m}+1$. Since $F$ is block-constant it is no restriction to assume that $n=m+1$, and hence we have to show that $$|F_m(\alpha)-F_{m+1}(\alpha)|\leq \|F\|_{Lip}(\alpha_m-\alpha_{m+1}).$$
If we let $\gamma$ be obtained from $\alpha$ by replacing the values on positions $m$ and $m+1$ with $\frac{\alpha_m+\alpha_{m+1}}{2}$, we have $F_m(\gamma)=F_{m+1}(\gamma)$, again using that $F$ is block-constant. Thus $$F_m(\alpha)-F_{m+1}(\alpha)= (F_m(\alpha)-F_m(\gamma))+(F_{m+1}(\gamma)-F_{m+1}(\alpha))=\scal{e_m-e_{m+1},F(\alpha)-F(\gamma)}.$$ Since $\|\alpha-\gamma\|=\frac{\alpha_m-\alpha_{m+1}}{\sqrt{2}}$ and the modulus of the right hand side can be estimated by $\|e_m-e_{m+1}\|\|F\|_{Lip}\|\alpha-\gamma\|$, the desired inequality follows.
\end{proof}

\begin{lemma}\label{rvq}
Let $F$ be a continuous block-constant vector-field in $\R_\geq^d$ and define $G:\R^d\rightarrow\R^d$ as the convolution $G(x)=F^{ext}*\Psi(x)$, where $\Psi(y)=\psi(y_1)\psi(y_2)\ldots \psi(y_d)$ and $\psi:\R\rightarrow \R$ is any $C^\infty$-function with compact support. Then $G$ is block-constant and $C^\infty$.
\end{lemma}
\begin{proof}
That $G$ becomes $C^\infty$ is a standard fact whose proof we omit. Note that $\Psi$ is permutation invariant, i.e.~$\Psi(\Pi x)=\Psi(x)$ for all permutations $\Pi$. We fix $x$ and consider $y$ as a variable. Given any point $x-y$ the value of $F^{ext}(x-y)$ is given by $F^{ext}(x-y)=\Sigma_{y}^*F(\Sigma_y(x-y))$ with $\Sigma_y\in \sort(x-y)$, and since $F$ is block-constant this value is independent of the particular choice of $\Sigma_y$ in case of ambiguity (i.e.~when $x-y$ is on the boundary of $\R^d_{\geq}$). Let $\Pi$ be an arbitrary perturbation. Then \begin{align*}&\Pi^*G(\Pi x)=\int_{\R^d}\Pi^*F^{ext}(\Pi x-y)\Psi(y)dy=\int_{\R^d}\Pi^*F^{ext}(\Pi (x-\Pi^*y))\Psi(\Pi^* y)dy=\\&\int_{\R^d}\Pi^*F^{ext}(\Pi (x-y))\Psi( y)dy=\int_{\R^d}\Pi^*\Upsilon_{\Pi y}^*F(\Upsilon_{\Pi y}\Pi (x-y))\Psi( y)dy\end{align*}
where each $\Upsilon_{\Pi y}$ is such that $\Upsilon_{\Pi y}\Pi (x-y)\in\R^d_\geq$. By the comments before the computation, we can consider $\Upsilon_{\Pi y}\Pi$ as $\Sigma_y$ and it follows that $\Pi$ has no effect on the outcome, so the above computation equals $\int_{\R^d}\Sigma_{y}^*F(\Sigma_y (x-y))\Psi( y)dy=G(x)$, as desired.
\end{proof}

We now come to the final theorem of the paper. This type of results can also be proved using the convex theory of complex sub-stochastic matrices, see \cite{andersson-etal-pams-2016}.

\begin{theorem}\label{t2}
Assume that $F$ is block-constant on $\R^d_\geq$ and Lipschitz continuous. Then \begin{equation}\label{lip1}\|\mathcal{L}_F(A)-\mathcal{L}_F(B)\|_2\leq \|F\|_{Lip}\|A-B\|_2.\end{equation}
\end{theorem}
\begin{proof}
It is east to see that $F^{ext}$ is continuous. By Lemma \ref{rvq} and a standard approximation argument, we may assume that $F$ is $C^\infty$ on $\R^d$ with the same Lipschitz constant. Moreover a simple matrix approximation argument shows that we may assume that $A$ has only simple eigenvalues. Set $E=A-B$ and let $U_t\Lambda_{\xi_t}U_t^*$ be the spectral decomposition of $B+tE$. By the discussion in Chapter II, Sec 1.1 \cite{kato2013perturbation}, we know that $\xi_t$ has distinct components for all $t$ except finitely many. Since $C^1$ vector-fields are automatically $C^1$ point-symmetric at all vectors with distinct components, $F$ is $C^1$ point-symmetric for all $\xi_t$ except finitely many values of $t$. By simply ignoring these points and appealing to Theorem \ref{t1} we get that
\begin{equation*}\begin{aligned}&\|\mathcal{L}_F(A)-\mathcal{L}_F(B)\|_2=\left\|\int_0^1 \frac{d}{dt}\mathcal{L}_F(B+tE)dt\right\|_2\leq \\
&\int_0^1\left\| U_t\Big([F,\xi_t]\circ \hat E+\diag_{(F'|_{\xi_t})^o \hat{e}} \Big)U_t^*\right\|_2dt\leq \sup_{0\leq t\leq 1}\left\| U_t\Big([F,\xi_t]
\circ \hat E+\diag_{(F'|_{\xi_t})^o \hat{e}} \Big)U_t^*\right\|_2.\end{aligned}\end{equation*} Note that $\hat E=U_t E U_t^*$ implicitly depends on $t$. As noted in the proof of Theorem \ref{t1} the supindex $o$ can be moved from $(F'|_{\xi_t})^o$ to $[F,\xi_t]$. Thus we get \begin{equation*}\begin{aligned}&\left\|U_t\big([F,\xi_t]\circ \hat E+\diag_{(F'|_{\xi_t})^o \hat{e}}\big)U_t^*\right\|_2^2=
\left\|[F,\xi_t]^o\circ \hat E+\diag_{(F'|_{\xi_t}) \hat{e}}\right\|_2^2\leq\\&\sum_{i\neq j}|[F,\xi_t]_{i,j}\hat E_{i,j}|^2+\| F'|_{\xi_t}\hat e\|^2_2=
\sum_{i\neq j}\|F\|_{Lip}^2|\hat E_{i,j}|^2+\| F\|_{Lip}^2\|\hat e\|^2_2=\| F\|_{Lip}^2\|\hat E\|^2_2
\end{aligned}\end{equation*} where we used Lemma \ref{rv}. Since $\|\hat E\|_2=\| E\|_2=\|A-B\|_2$, the theorem follows by inserting this estimate in the above supremum.
\end{proof}

The assumption that $F$ is block-constant in all of $\R^d_\geq$ is crucial for the above result to be true. To see this consider the case when $\mathcal{L}_F$ is the orthogonal projection onto the first eigenspace and $n=2$ say. The matrices $\left(
                                          \begin{array}{cc}
                                            1\pm\varepsilon & 0 \\
                                            0 & 1\mp\varepsilon \\
                                          \end{array}
                                        \right)$
then show that $\mathcal{L}_F$ is not continuous, despite $F(x)=(1,0)$ being constant (and hence having Lipschitz constant 0).

\bibliographystyle{plain}
\bibliography{MCref}
\end{document}